\documentclass[
submission
]{dmtcs-episciences}

\usepackage[utf8]{inputenc}
\usepackage{subfigure}

\usepackage[round]{natbib}
\usepackage{amssymb,amsmath,bm}
\usepackage{caption}
\usepackage{constants,color}
\usepackage{enumerate}
\usepackage{bbm}
\usepackage{appendix}
\allowdisplaybreaks
\newcommand{\bea}{\begin{eqnarray}}
\newcommand{\ena}{\end{eqnarray}}
\newcommand{\beas}{\begin{eqnarray*}}
\newcommand{\enas}{\end{eqnarray*}}
\newcommand{\beq}{\begin{equation}}
\newcommand{\enq}{\end{equation}}
\newcommand{\ignore}[1]{}

\newtheorem{theorem}{Theorem}[section]

\newtheorem{lemma}{Lemma}[section]

\newtheorem{remark}{Remark}[section]

\author{\"{U}m\.{i}t I\c{s}lak \affiliationmark{1}
  \and Alperen Y. \"{O}zdemir \affiliationmark{2}}
  
\title[On an Alternative Sequence Comparison Statistic of Steele]{On an Alternative Sequence Comparison Statistic of Steele}
\affiliation{
  Bo\u{g}azi\c{c}i University, Department of Mathematics, Istanbul, Turkey\\
  University of Southern California, Department of Mathematics, Los Angeles, California}

\received{2019-9-4}

\revised{2020-4-21}

\accepted{2020-4-25}

\begin{document}

\publicationdetails{22}{2020}{1}{18}{5745}

\maketitle
\begin{abstract}
  The purpose of this paper is to study a statistic that is used to compare the similarity between two strings, which is first introduced by Michael Steele in 1982. It was proposed as an alternative to the length of the longest common subsequences, for which the variance problem is still open. Our results include moment asymptotics and  distributional asymptotics for Steele's statistic and a variation of it in random words.
\end{abstract}

\section{Introduction}\label{sec:intro}

The most well-known approach in  sequence comparison is the use of the longest common subsequences. This is  related partially to its wide range applications in various field such as computational biology, computer science and bioinformatics, and partially to the challenges that it presents in theory. 
By definition, $LC_n$, the length of the longest common subsequences  of
 sequences $X_1\cdots X_n$ and $Y_1\cdots Y_n$, is the maximal
integer $k\in [n] := \{1,\dots, n\}$, such that there exist $1\le i_1<
\cdots < i_k\le n$ and $1\le j_1<\cdots <j_k\le n$, such that
$$
X_{i_{\ell}}=Y_{j_{\ell}} \quad \quad {\rm for \quad all} \quad \quad
\ell=1,2,\dots, k.
$$
The theory  of $LC_n$ has a long history starting with the well-known
result of Chv\'atal and Sankoff \cite{CS}. They show that if $X_i$'s and $Y_j$'s are independent and identically distributed (i.i.d.) discrete random variables, and if the sequences are independent among themselves,
\begin{equation*}\label{lcmean}
\lim_{n \to \infty}\frac{\mathbb{E}LC_n}{n}=\gamma_m^*,
\end{equation*}
where $\gamma_m^*$ is some constant in $[0,1]$. 
To this day, the exact value of $\gamma_m^*$ (which depends
on the distribution of $X_1$ and on the size of the alphabet) is
unknown, even in  the simplest case where one has  uniform Bernoulli
random variables. Furthermore, the order of $Var(LC_n)$  and the asymptotic distribution of $LC_n$ are  still unknown for uniform Bernoulli
random variables. We refer   \cite{LM2009}  and \cite{HoudreIslak:2014} to the reader for  some  recent progress towards these problems. The former of these two shows that the variance of $LC_n$ in random words is of order $n$  under certain asymmetry conditions, and the latter one proves that the same conditions yield a central limit theorem after proper centering and scaling. 
Also see \cite{GHI:2015},  \cite{HoudreIslak:2014} and \cite{Jin}   for recent results  similar to the   ones mentioned    for a score function setting, for independent uniformly  random permutations and for Mallows permutations, respectively. 

There had been various alternatives to the longest common subsequences for sequence comparison where some of the technical difficulties $LC_n$  do not emerge. Most of these rely on comparison of words based on matching   of subsequences of given two or more sequences. As two general references sequence comparisons and word statistics, we refer to 
\cite{W:84} and  \cite{RSW:00}. More specifically,  Waterman in \cite{W:86}  studies the longest match of two sequences interrupted by at most $k$ dismatches.
\cite{GW:92} is on  a sequence comparison test  based on $k$-word matches on  a diagonal of a sequence comparison. In particular, they require that at least $i$ of the $k$ letters of the words to match where $ i \leq k$, and they provide Poisson approximations for certain statistics. 
 \cite{RW:06} studies length of the longest exact match of a random sequence across another sequence  again in terms of distributional approximations.    More recently, Reinert et al., \cite{Retal09},  assume that $\mathbf{A} = A_1 \ldots A_{n_1}$ and  $\mathbf{B} =B_1 \ldots B_{n_2}$ are words where the letters are from a finite alphabet $\mathcal{A}$ of size $m$.  For $\mathbf{w} = (w_1,\ldots,w_k) \in \mathcal{A}^k$, they define $$X_{\mathbf{w} } = \sum_{i = 1}^{n_1 - k +1} \mathbf{1}(A_i = w_1,\ldots, A_{i+k - 1} = w_k)$$  which counts the number of occurences of $\mathbf{w} $ in $\mathbf{A}$. Similarly, letting $Y_{\mathbf{w}}$   count the number of occurences of $\mathbf{w} $ in $\mathbf{B}$, the sequence comparison statistic is defined by $$D = \sum_{\mathbf{w} \in \mathcal{A}^k} X_{\mathbf{w}} Y_{\mathbf{w}}.$$    Afterwards, the authors study hypothesis testing based on this statistic. Also, see the continuation work \cite{Wanetal}.

We study yet another sequence comparison statistic which was proposed by M. Steele in 1982 during his investigations on the longest common subsequence problem \cite{Steele1982}.  This also compares matchings of subsequences - but this time involving all subsequences possible. Namely, letting $X_1,\cdots,X_n$ and $Y_1,\cdots,Y_n$ be uniformly distributed over a finite alphabet, the statistic of Steele is given by $$T_n = \sum_{k=1}^n T_{n,k},$$ where 

\begin{equation} \label{teneke}
T_{n.k} =  \sum_{1 \leq i_1 < \cdots < i_k \leq n} \sum_{1 \leq j_1 < \cdots < j_k \leq n}  \mathbf{1} (X_{i_1} = Y_{j_1},\ldots, X_{i_k} = Y_{j_k}).
\end{equation}
The purpose of this paper is to analyze  $T_{n,k}$  in terms of their moment asymptotics and to show that a central limit theorem holds for $T_{n,k}$ when $k$ is kept fixed. The ultimate objective is to study $T_n,$ which is relatively difficult, so is postponed to a future paper. 

Let us now fix some notation for the following sections. First, $=_d$, $\rightarrow_d$ and
$\rightarrow_{\mathbb{P}}$ are used for equality in distribution,
convergence in distribution and convergence in probability,
respectively. $\mathcal{G}$ denotes a standard normal random
variable, and $d_K$ is used for Kolmogorov distance between probability measures. Finally, for two
sequences $a_n, b_n$, we write $a_n \sim b_n$ for $\lim_{n
\rightarrow \infty} a_n / b_n =1$.

The rest of the paper is organized as follows. In Section \ref{sec:randomwordscase}, we identify bounds on the first two moments of the statistics aforementioned. Then, in Section \ref{sec:CLT}, a central limit theorem for $T_{n,k}$ is proven. The next section deals with the computation time for $T_{n,k}$ and shows that the corresponding  running time  is $\Theta(k n^2)$.  In Section \ref{sec:alphabet}, we analyze the first moment asymptotics of $T_{n}$ with respect to different sizes of the alphabet.

\section{Moments of  $T_{n,k}$ and $T_n$} \label{sec:randomwordscase}
Unless otherwise mentioned, from here on the random variables of the independent sequences $X_1, \dots, X_n$ and $Y_1,\dots, Y_n$ are i.i.d. with common distribution that is uniform over a finite alphabet of size $a.$ Our first result concerning the statistics \eqref{teneke} is as follows.
\begin{theorem}\label{momTnk}
Let    $k \in \mathbb{N}$ be fixed.  We have  $$\mathbb{E}[T_{n,k}] = \binom{n}{k}^2 \frac{1}{a^{k}} \sim \frac{n^{2k}}{(k!)^2 a^{k}},$$
and  \begin{equation}\label{ETnksquarebounds}
\binom{n}{k}^4   \frac{1}{a^{2k}} \leq \mathbb{E}[T_{n,k}^2] \leq  \binom{n}{k}^2  \frac{1}{a^{k}} \,  \sum_{j=0}^k \binom{n-k}{j}^2 \binom{n -j }{k-j}^2  \frac{1}{a^{j}}.
\end{equation}
Moreover,  the lower bound in \eqref{ETnksquarebounds} satisfies $$\binom{n}{k}^4   \frac{1}{a^{2k}}  \sim \frac{1}{(k!)^4a^{2k}} n^{4k}, \qquad n \rightarrow \infty,$$ and the upper bound in  \eqref{ETnksquarebounds} satisfies 
$$ \sum_{j=0}^k \binom{n-k}{j}^2 \binom{n}{k-j}^2  \frac{1}{a^{j}} \sim  \left(\frac{1}{(k!)^2 a^{k}}  \sum_{j=0}^k  \frac{1}{(j!)^2 ((k-j)!)^2 a^j}  \right) n^{4k}, \qquad n \rightarrow \infty.$$
\end{theorem} 

\begin{proof}
The expectation formula and the lower bound for the second moment are straightforward. We study the upper bound for the second moment with the expression
$$\mathbb{E}[T_{n,k}^2]  = \sum_{(\mathcal{I}_1,\mathcal{I}_2), (\mathcal{I}_1',\mathcal{I}_2')}  \mathbb{E}\left[\mathbf{1}\left(\bigcap_{i_s \in \mathcal{I}_1, j_s \in \mathcal{I}_2} \{X_{i_s} = Y_{j_s} \} \right) \mathbf{1}\left(\bigcap_{i_s' \in \mathcal{I}_1', j_s' \in \mathcal{I}_2'}\{ X_{i_s'} = Y_{j_s'} \} \right)\right],$$ where the summation $ \sum_{(\mathcal{I}_1,\mathcal{I}_2), (\mathcal{I}_1',\mathcal{I}_2')}$ is taken over  all subsets  $\mathcal{I}_1,\mathcal{I}_2,\mathcal{I}_1',\mathcal{I}_2'$ of $[n]$ each of which has cardinality $k$.  We can rewrite this expression as $$\mathbb{E}[T_{n,k}^2]  = \sum_{(\mathcal{I}_1,\mathcal{I}_2), (\mathcal{I}_1',\mathcal{I}_2')} \mathbb{E}[ \chi(\mathcal{I}_1,\mathcal{I}_2) \chi(\mathcal{I}_1',\mathcal{I}_2')],$$ where $\chi(\mathcal{I}_1,\mathcal{I}_2)$ is the indicator of the event  $\bigcap_{i_s \in \mathcal{I}_1, j_s \in \mathcal{I}_2} \{X_{i_s} = Y_{j_s}\}$. Then define  $(\mathcal{I}_1', \mathcal{I}_2') \ominus (\mathcal{I}_1,\mathcal{I}_2)$ to be  the set 
\begin{equation*}
\{(i_s', j_s') \in (\mathcal{I}_1', \mathcal{I}_2'): i_s' \notin \mathcal{I}_1 \; \text{and} \; j_s' \notin \mathcal{I}_2\}.
\end{equation*}  To give a simple example of the set defined above, take $n=3,\, k=2$ and define $\mathcal{I}_1=\{1,2\}, \mathcal{I}_2=\{1,3\}, \mathcal{I}_1'=\{1,3\}$ and $\mathcal{I}_2'=\{1,2\}.$ Then,  $(\mathcal{I}_1', \mathcal{I}_2') \ominus (\mathcal{I}_1,\mathcal{I}_2)$ is  $\{(3,2)\}.$   Clearly, $$\chi(\mathcal{I}_1', \mathcal{I}_2') \leq \chi ((\mathcal{I}_1', \mathcal{I}_2')  \ominus (\mathcal{I}_1,\mathcal{I}_2)),$$ and $\chi((\mathcal{I}_1', \mathcal{I}_2') \ominus (\mathcal{I}_1,\mathcal{I}_2))$ is independent of $\chi(\mathcal{I}_1, \mathcal{I}_2)$.  Therefore 
\begin{eqnarray*}
\mathbb{E}[T_{n,k}^2]  &=& \sum_{(\mathcal{I}_1,\mathcal{I}_2), (\mathcal{I}_1',\mathcal{I}_2')} \mathbb{E}[ \chi(\mathcal{I}_1,\mathcal{I}_2) \chi(\mathcal{I}_1',\mathcal{I}_2')] \\
&\leq&  \sum_{(\mathcal{I}_1,\mathcal{I}_2), (\mathcal{I}_1',\mathcal{I}_2')} \mathbb{E}[ \chi(\mathcal{I}_1,\mathcal{I}_2) \chi ((\mathcal{I}_1', \mathcal{I}_2')  \ominus (\mathcal{I}_1,\mathcal{I}_2))] \\ 
&=&  \sum_{(\mathcal{I}_1,\mathcal{I}_2), (\mathcal{I}_1',\mathcal{I}_2')} \mathbb{E}[\chi(\mathcal{I}_1,\mathcal{I}_2)] \mathbb{E}[ \chi ((\mathcal{I}_1', \mathcal{I}_2')  \ominus (\mathcal{I}_1,\mathcal{I}_2))] \\
&=& \sum_{(\mathcal{I}_1, \mathcal{I}_2)} \frac{1}{a^{k}} \sum_{(\mathcal{I}_1',\mathcal{I}_2')}  \mathbb{E}[ \chi ((\mathcal{I}_1', \mathcal{I}_2')  \ominus (\mathcal{I}_1,\mathcal{I}_2))]  \\
&=& \sum_{(\mathcal{I}_1, \mathcal{I}_2)} \frac{1}{a^{k}} \sum_{j=0}^k  \sum_{|(\mathcal{I}_1',\mathcal{I}_2')  \ominus (\mathcal{I}_1, \mathcal{I}_2)| = j}  \mathbb{E}[ \chi ((\mathcal{I}_1', \mathcal{I}_2')  \ominus (\mathcal{I}_1,\mathcal{I}_2))]  \\
&\leq &  \sum_{(\mathcal{I}_1, \mathcal{I}_2)} \frac{1}{a^{k}} \, \sum_{j=0}^k \binom{n-k}{j}^2\left(\sum_{\ell = 0}^{k - j} \binom{k }{\ell} \binom{n - k - j}{k -\ell-j}\right)^2 \frac{1}{a^{j}} \\
&=& \binom{n}{k}^2  \frac{1}{a^{k}} \,  \sum_{j=0}^k \binom{n-k}{j}^2 \binom{n-j}{k-j}^2  \frac{1}{a^{j}}.
\end{eqnarray*}
The asymptotics for the lower bound  of $\mathbb{E}[T_{n,k}^2]$ is immediate from the Stirling formula. For the upper bound, we observe that 
 \begin{align*}
&\hspace{-0.5in} \binom{n}{k}^2  \frac{1}{a^{k}} \,  \sum_{j=0}^k \binom{n-k}{j}^2 \binom{n - j}{k-j}^2  \frac{1}{a^{j}}   \\ & \sim  \frac{n^{2k}}{(k!)^2 a^k} \sum_{j=0}^k \frac{1}{(j!)^2 ((k-j)!)^2} \frac{((n-j)!)^2 ((n-k)!)^2 }{((n -k -j)!)^2 ((n -k)!)^2 } \frac{1}{a^j} \\
& = \frac{n^{2k}}{(k!)^2 a^k} \sum_{j=0}^k \frac{1}{(j!)^2 ((k-j)!)^2} \left( \frac{(n-j)!   }{(n -k -j)!    }\right)^2 \frac{1}{a^j} \\ 
& \sim    \frac{n^{2k}}{(k!)^2 a^{k}} \sum_{j=0}^k  \frac{1}{(j!)^2 ((k-j)!)^2 a^j}  n^{2k}  \\
& = \left(\frac{1}{(k!)^2 a^{k}}  \sum_{j=0}^k  \frac{1}{(j!)^2 ((k-j)!)^2 a^j}  \right) n^{4k}.
\end{align*}

\end{proof}
 
Theorem  \ref{momTnk} shows that the order of  $\mathbb{E}[T_{n,k}^2]$ is $n^{4k}$. Beyond this, the exact computation of the second moment looks quite involved, and we intend to analyze it in a subsequent work.  An even more challenging work would be to study the moments when $k$ grows along with $n$.

\begin{remark} (i.)
Certain results in this paper can also be generalized to non-uniform random words.  For example, if the independent sequences $X_1,\ldots,X_n$ and $Y_1,\ldots,Y_n$ consist of  i.i.d. random variables with support $[a]$ and with distributions $p_j = \mathbb{P}(X_1 =j)$, we  may   as before  define $$T_{n,\mathbf{p}} = \sum_{k=1}^n \sum_{1 \leq i_1 < \cdots < i_k \leq n} \sum_{1 \leq j_1 < \cdots < j_k \leq n}  \mathbf{1} (X_{i_1} = Y_{j_1},\ldots, X_{i_k} = Y_{j_k}).$$ In this case, one easily obtains $$  \mathbb{E}[T_{n,\mathbf{p}}] =\sum_{k=0}^n \binom{n}{k} \left( \sum_{j=1}^a p_j^2\right)^k.$$ 
It is also possible to give a similar computation for $ \mathbb{E}[(T_{n,\mathbf{p}})^2]$, but this time it will be more complicated in terms of notation.  Below we restrict ourselves to the uniform case due to keeping notational simplicity and due to  the fact that asymptotic computations will require certain growth  assumptions on  $\mathbf{p}$, which we do not think that will contribute to the  gist of the paper. 
\end{remark}

\section{Central limit theorem for $T_{n,k}$} \label{sec:CLT}

Followig the moment calculations and estimates, we prove a central limit theorem for \eqref{teneke} in this section.
\begin{theorem}\label{thm:CLTTnk}
We have
$$\frac{T_{n,k} - \mathbb{E}[T_{n,k}]}{\sqrt{Var(T_{n,k})}} \longrightarrow_d \mathcal{G}, \qquad n \rightarrow \infty,$$ where 
$\mathbb{E}[T_{n,k}] = \binom{n}{k}^2 \frac{1}{a^k}$, and $\mathbb{E}[T_{n,k}^2]$ satisfies the bounds in Theorem \ref{momTnk}.
\end{theorem}

\begin{proof}
We start with an observation that
 for any $r \in [n]$,
\begin{equation*}
\mathbf{1}(X_r = Y_r)   = \frac{\mathbf{1}(X_r \geq Y_r) + \mathbf{1}(X_r \leq Y_r) - \mathbf{1}(X_r > Y_r) - \mathbf{1}(X_r < Y_r)}{2}.
\end{equation*}
This implies that 
\begin{align*}
& \mathbf{1}(X_{i_1} = Y_{j_1},\ldots, X_{i_k} = Y_{i_k})  \\ & \hspace{0.25in} = 
\frac{1}{2^k} \prod_{s=1}^k  \left(\mathbf{1}(X_{i_s}  
 \geq Y_{j_s}) + \mathbf{1}(X_{i_s} \leq Y_{j_s}) - \mathbf{1}(X_{i_s} > Y_{j_s}) - \mathbf{1}(X_{i_s} < Y_{j_s})\right).
\end{align*}

 Now let $\{U_i\}_{i \in \mathbb{N}}$ and $\{V_i\}_{i \in \mathbb{N}}$ be i.i.d. random variables that are uniformly distributed over $(0,1)$,  and  define permutations $\sigma$ and $\gamma$ in $S_n$ so that 
\begin{equation}\label{uv}
U_{\sigma(1)} < \cdots <  U_{\sigma(n)} \qquad \text{and} \qquad V_{\gamma(1)} < \cdots < V_{\gamma(n)}.
\end{equation}
Further, let us define
 $$h(i_1,\ldots,i_k;j_1,\ldots,j_k) = \frac{1}{2^k} \prod_{s=1}^k  \left(\mathbf{1}(X_{i_s} \geq Y_{j_s}) + \mathbf{1}(X_{i_s} \leq Y_{j_s}) - \mathbf{1}(X_{i_s} > Y_{j_s}) - \mathbf{1}(X_{i_s} < Y_{j_s})\right),$$ 
 and $$\mathcal{S}_{1} = \{(i_1,\ldots,i_k), (j_1,\ldots,j_k): 1 \leq i_1 < \cdots < i_k \leq n, 1 \leq j_1 < \cdots < j_k \leq n\}$$ so that
$$T_{n,k} =_d \sum_{\mathcal{S}_1} h(i_1,\ldots,i_k;j_1,\ldots,j_k).$$ Now we observe that 
\begin{align*}
T_{n,k} =_d&  \sum_{\mathcal{S}_1} h(\sigma(i_1),\ldots,\sigma(i_k);\gamma(j_1),\ldots,\gamma(j_k))\\
=& \sum_{\mathcal{S}_2} h(\sigma(i_1),\ldots,\sigma(i_k);\gamma(j_1),\ldots,\gamma(j_k))  
\mathbf{1}(1 \leq i_1 < \cdots < i_k \leq n) \mathbf{1}(1 \leq j_1 < \cdots < j_k \leq n),
\end{align*}
where we set $$\mathcal{S}_2 = \{(i_1,i_2, \ldots, i_k): i_r \in [n], r =1,\ldots,k\}.$$ So, by \eqref{uv}
\begin{eqnarray*}
T_{n,k} &=_d&   \sum_{\mathcal{S}_2} h(\sigma(i_1),\ldots,\sigma(i_k);\gamma(j_1),\ldots,\gamma(j_k))  
\mathbf{1}(U_{\sigma(i_1)} < \cdots < U_{\sigma(i_k)})\\
&& \times \mathbf{1}(V_{\gamma(j_1)} < \cdots < V_{\gamma(j_k)}) \\
&=&  \sum_{\mathcal{S}_2} h(i_1,\ldots,i_k; j_1,\ldots,j_k)  
\mathbf{1}(U_{i_1} < \cdots < U_{i_k})  \mathbf{1}(V_{j_1} < \cdots < V_{j_k}).
\end{eqnarray*}
Now, define  
\begin{eqnarray*}
f((x_{i_1}, y_{i_1}, u_{i_1}, v_{i_1}), \ldots, (x_{i_k}, y_{i_k}, u_{i_k}, v_{i_k})) &=&  h(i_1,\ldots,i_k; j_1,\ldots,j_k)  
\mathbf{1}(u_{i_1} < \cdots < u_{i_k}) \\
&& \times   \mathbf{1}(v_{j_1} < \cdots < v_{j_k} )
\end{eqnarray*}
and $$g((x_{i_1}, y_{i_1}, u_{i_1}, v_{i_1}), \ldots, (x_{i_k}, y_{i_k}, u_{i_k}, v_{i_k})) = \sum f((x_{i_1}, y_{i_1}, u_{i_1}, v_{i_1}), \ldots, (x_{i_k}, y_{i_k}, u_{i_k}, v_{i_k})),$$ where the summation on right-hand side is over all $(i_1,\ldots,i_k) \in S_{i_1,\ldots,i_k}$ and $(j_1,\ldots,j_k) \in S_{j_1,\ldots,j_k}$ with $S_{i_1,\ldots,i_k}$ and $S_{j_1,\ldots,j_k}$ being all permutations of $i_1,\ldots,i_k$ and $j_1,\ldots,j_k$, respectively. 

Then we arrive at $$T_{n,k} =_d \sum_{\mathcal{S}_1} g((X_{i_1},Y_{i_1},U_{i_1},V_{i_1}), \ldots, (X_{i_k}, Y_{i_k}, U_{i_k}, V_{i_k})),$$ which is  recognized to be a $U$-statistic noting that  (i) $g$ is symmetric,
(ii)  $g$ is a function of random vectors whose coordinates are
  independent,  and that 
(iii)  $g \in L^2$. We need the following result of Chen and Shao to conclude the proof. 
\end{proof}

\begin{theorem}\label{thm:CSUstats}\cite{CS2007}
Let $X_1,\ldots,X_n$ be i.i.d. random variables, $\xi_n$ be a $U$-statistic
with symmetric kernel $g$, $\mathbb{E}[g(X_1,...,X_m)]=0,
\sigma^2=Var(g(X_1,\ldots,X_m))<\infty$ and
$\sigma_1^2=Var(g_1(X_1))>0.$  If in addition
$\mathbb{E}|g_1(X_1)|^3< \infty,$ then $$d_K \left(\frac{\sqrt{n}}{m
\sigma_1}\xi_n, \mathcal{G} \right)  \leq \frac{6.1 \,
\mathbb{E}|g_1(X_1)|^3}{\sqrt{n} \sigma_1^3}
+\frac{(1+\sqrt{2})(m-1)\sigma}{(m(n-m+1))^{1/2} \sigma_1}.
$$ 
\end{theorem}
Now, recalling the fact \cite{Lee} that $(m \sigma_1) / \sqrt{n} \sim \sqrt{Var(T_{n,k})}$, and using Slutsky's theorem   we conclude that $$\frac{T_{n,k} - \mathbb{E}[T_{n,k}]}{ \sqrt{Var(T_{n,k})}} \longrightarrow_d \mathcal{G}, \qquad n \rightarrow \infty,$$ as required.  The growth rate of $Var(T_{n,k})$ as a function of $n$ is not known to the authors, but the above theorem does not impose any condition on it. Note that one may further obtain convergence rates via Theorem \ref{thm:CSUstats}, but we do not go into details of this here.

\section{Computation time for $T_{n,k}$}
We are thankful to Michael Waterman, who provided us with the following algorithm for computing the number of $k$-long common subsequences of two random words. The algorithm uses dynamic programming similar to the case of finding the length of the longest common subsequence, which has a running time of $\Theta(n^2)$ \cite{W}. In our case, it requires $\Theta(kn^2)$ operations. A description of it is as follows.\\

First, we define 
\begin{equation*}
S_{l}(i,j)=\mathbf{1}(X_i=Y_j) \sum_{ i_1 < \cdots < i_l < i} \sum_{ j_1 < \cdots < j_l < j}  \mathbf{1} (X_{i_1} = Y_{j_1},\ldots, X_{i_l} = Y_{j_l}),
\end{equation*}
where $1\leq l \leq k.$ $S_l(i,j)$ counts the number of $l$-long subsequences ending exactly at i$th$ and j$th$ positions of the first and the second sequences respectively. Recursively, 
\begin{equation*}
S_{l}(i,j)=\mathbf{1}{(X_i=Y_j)} \big\{ \sum_{\substack{\alpha < i \\ 
\beta < j}} S_{l-1}(\alpha, \beta) + \sum_{\alpha < i} S_{l-1}(\alpha,j) + \sum_{\beta < j} S_{l-1}(i,\beta) \big\}. 
\end{equation*}
Next, define $T_l(i,j) = \sum_{\substack{\alpha \leq i \\ 
\beta \leq j}} S_{l-1}(\alpha, \beta), \,$ 
$C_l(i,j) = \sum_{\alpha \leq i} S_{l-1}(\alpha,j), \,$ and 
$R_l(i,j) =\sum_{\beta \leq j} S_{l-1}(i,\beta).$
It is easy to see that they satisfy the recursive relations below.
\begin{align} \label{rec1}
T_l(i,j) &=T_{l}(i\text{-}1,j\text{-}1)+C_l(i\text{-}1,j)+R_{l}(i,j\text{-}1)+S_l(i,j),  \notag \\
C_l(i,j) &= C_l(i\text{-}1,j) + S_l(i,j), \\
R_l(i,j) &= R_l(i,j\text{-}1)+ S_l(i,j). \notag
\end{align}
Then we can rewrite our counting function as
\begin{equation}\label{rec2}
S_{l}(i,j)=\mathbf{1}{(X_i=Y_j)} \big\{  T_{l\text{-}1}(i\text{-}1,j\text{-}1)+R_{l\text{-}1}(i,j\text{-}1)+C_{l\text{-}1}(i\text{-}1,j) \big\}.
\end{equation}

As we proceed from $l$ to $l+1$ through the algorithm, we need to go through \eqref{rec1} and \eqref{rec2} for all $(i,j),$ which requires a constant (independent of $n$ and $k$) times $n^2$ operations. The total number of $k$-long common subsequences is given by the largest $S_k(i,j),$ which is increasing both in $i$ and $j$ unless it is zero. Therefore, the total running time to compute $T_{n,k}$ is $\Theta(kn^2).$ Since we need a constant times $n^2$ operations to obtain $T_{n,k+1}$ once we run the algorithm for $T_{n,k},$ the running time is $\Theta(n^3)$ to compute $T_n.$

\section{Asymptotics of $\mathbb{E}[T_n]$ for growing alphabet} \label{sec:alphabet}

An immediate corollary to Theorem \ref{momTnk}  if $a$ is a fixed number is   $$\mathbb{E}[T_n] = \sum_{k=1}^n \mathbb{E}[T_{n,k}]  =   \sum_{k=1}^n \binom{n}{k}^2 \frac{1}{a^k}.$$ Our purpose in this section is to see the effect of changing $a$ along with $n$ as $n \rightarrow \infty$. Results of this section are summarized as follows. 

\begin{theorem}\label{thm:ETnasypmtotics}
Let $a_n=an^{\alpha}$ be the size of the alphabet where $n$ is the length of the sequences and $a, \alpha$ be positive constants. Define $k^* = \frac{n}{1  + \sqrt{a n^{\alpha}}}$.Then, as $n\rightarrow \infty$, the asymptotic behavior of $\mathbb{E}[T_n]$ with respect to $a_n$ is summarized in the table below.

\vspace{0.11in}

\renewcommand{\arraystretch}{1.8}
\begin{center}
\begin{tabular}{|c|c|}
\hline
$\alpha(a_n=an^{\alpha})$ & $\mathbb{E}[T_n]= \sum_{k=1}^{n}\binom{n}{k}^2 \frac{1}{a_n^{k}} \sim$  \\ \hline 
$0$ &   $\frac{\sqrt[4]{a}}{2 \sqrt{\pi n}} \Big(1+\frac{1}{\sqrt{a}} \Big)^{2n+1}$   \\  \hline
$(0,1/2)$ &  $
\frac{\sqrt[4]{an^{\alpha}}}{2 \sqrt{\pi n}} e^{-\frac{k^*{^2}}{2n} (1  + o(1)) } e^{\frac{2n}{1+ \sqrt{an^{\alpha}}}} \Big(1+ \frac{1}{\sqrt{an^{\alpha}}} \Big)^{\frac{2n}{1+ \sqrt{an^{\alpha}}}}$  \\ \hline   
$[1/2, 2/3)$  & $
\frac{\sqrt[4]{an^{\alpha}}}{2 \sqrt{\pi n}} e^{-\frac{k^*{^2}}{2n} -\frac{k^*{^3}}{6n^2}} e^{\frac{2n}{1+ \sqrt{an^{\alpha}}}} \Big(1+ \frac{1}{\sqrt{an^{\alpha}}} \Big)^{\frac{2n}{1+ \sqrt{an^{\alpha}}}}$  \\  \hline
$[2/3, 1)$  & $\frac{\sqrt[4]{an^{\alpha}}}{2 \sqrt{\pi n}} e^{\frac{2}{\sqrt{a}}n^{1- \alpha/2} - \frac{1}{2} \frac{1}{1+\sqrt{an^{\alpha}}}} $ \\  \hline
$1$ &   $\frac{\sqrt[4]{a}}{2 \sqrt{\pi}} e^{\frac{3}{2a}} n^{-1/4} e^{\frac{2}{\sqrt{a}} n ^{1/2}}$  \\  \hline
$(1,2)$  & $\frac{\sqrt[4]{an^{\alpha}}}{2 \sqrt{\pi n}} e^{\frac{2}{\sqrt{a}}n^{1- \alpha/2}}$  \\ \hline \hline
Unif. Perm. & $\frac{1}{2 \sqrt{\pi e}}n^{-1/4} e^{2 n^{1/2}}$ \textnormal{\cite{LP1981}} \\ \hline
\end{tabular} 
\end{center}
\footnote{The last line gives the asymptotic behavior of $T_n((\pi_1,\dots,\pi_n),(\sigma_1,\dots,\sigma_n))$ where $\pi=(\pi_1,\dots, \pi_n)$ and $\sigma=(\sigma_1,\dots, \sigma_n)$ are uniform random permutations of the set $[n]$.}
\end{theorem}

\begin{proof}
The proof uses the technique in Chapter 5 of \cite{S14}, which is used for sums of binomial coefficient powers therein. In our case, the sum includes also an exponential term, which yields asymmetric distribution of terms around the maximum term unlike the binomial coefficient only case. But the terms to the right and the terms to the left to the maximum term are dealt in the same manner as shown below. \\

We start with locating the maximum term of the sum, then evaluate the sum of the other terms with respect to the maximum term. \\

We first observe that the ratio of two consecutive terms is
\begin{equation*}
\binom{n}{k+1}^2 \frac{1}{a_n^{k+1}} \Big/ \binom{n}{k}^2 \frac{1}{a_n^k} = \Big(\frac{n-k}{k+1} \Big)^2 \frac{1}{a_n}.
\end{equation*}
Since the ratio is monotone decreasing, the sequence of terms in the sum is unimodal. The maximum term occurs for the first $k$ where the fraction above is less than one. Observe that 
\begin{align*}
(n-k)^2 < a_n (k+1)^2 \Leftrightarrow&   (n-k) < \sqrt{a_n} (k+1) \\
\Leftrightarrow& \frac{1}{1+\sqrt{a_n}}n - \frac{\sqrt{a_n}}{1+\sqrt{a_n}} < k, \\
\Leftrightarrow&  k_{max}\in \Big[\frac{1}{1+\sqrt{a_n}}n - \frac{\sqrt{a_n}}{1+\sqrt{a_n}}, \frac{1}{1+\sqrt{a_n}}n + \frac{1}{1+\sqrt{a_n}} \Big].
\end{align*}
Let $k^*=\frac{1}{1+\sqrt{a_n}}n$, which lies in the same interval with $k_{max}.$ Since we are interested only in the asymptotics of  $\binom{n}{k_{max}},$ it is justified to work with $\binom{n}{k^*}.$ \\

Then, we consider the remaining terms. First we take the higher indexed terms, namely $k>k^*.$ 
Let $k=k^* +i,$ $i > 0$,  and also set $A_n = \frac{(1+\sqrt{a_n})^2}{\sqrt{a_n}}$. Referring to the method discussed in Chapter 5 of  \cite{S14}, defining $$R :=  \frac{\binom{n}{k}^2 \frac{1}{a_n^{k}}}{\binom{n}{k^*}^2 \frac{1}{a_n^{k^*}}},$$ we have
$$
R 
= \frac{\binom{n}{k^*+i}^2 \frac{1}{a^{k^*+i}}}{\binom{n}{k^*}^2 \frac{1}{a^{k^*}}} 
 =  a_n^{-i} \frac{(n-k^*)^2 \cdots (n-k^*-i+1)^2}{(k^*+i)^2 \cdots (k^*+1)^2}.
$$
It follows that
\begin{align*}
\ln R =& -i \ln a_n +2 \sum_{j=1}^i \ln \left(\frac{n-k^*-j+1}{k^*+j} \right) \\
=& 2 \sum_{j=1}^i \left( \ln \left(\frac{n-k^*-j+1}{k^*+j} \right) - \ln \sqrt{a_n} \right) \\
=& 2 \sum_{j=1}^i \left( \ln \left(\frac{n-\frac{1}{1+\sqrt{a_n}}n-j+1}{\frac{1}{1+\sqrt{a_n}}n+j} \right) - \ln \sqrt{a_n} \right) \\
=& 2 \sum_{j=1}^i  \ln \left(\frac{\frac{\sqrt{a_n}}{1+\sqrt{a_n}}n-j+1}{\frac{\sqrt{a_n}}{1+\sqrt{a_n}}n+\sqrt{a_n}j} \right)  \\
=& 2 \sum_{j=1}^i  \ln \left(1- \frac{(1+\sqrt{a_n})j-1}{\frac{\sqrt{a_n}}{1+\sqrt{a_n}}n+\sqrt{a_n}j}\right) \\
=& 2 \sum_{j=1}^i  \ln \left(1-A_n \frac{j}{n+ (1+\sqrt{a_n})j} + \frac{1+\sqrt{a_n}}{\sqrt{a_n}} \frac{1}{n+ (1+\sqrt{a_n})j}\right) \\
=& 2 \sum_{j=1}^i  \left(-A_n \frac{j}{n+ (1+\sqrt{a_n})j} + \frac{1+\sqrt{a_n}}{\sqrt{a_n}} \frac{1}{n+ (1+\sqrt{a_n})j} +\Theta(\frac{A_n^2 j^2}{n^2}) \right) \\
 \sim & -2A_n \frac{i(i+1)}{2n} + 2 \frac{1+\sqrt{a_n}}{\sqrt{a_n}} \frac{i}{n} +\Theta(\frac{A_n^2 i^3}{n^2}) \\
=& \frac{-A_n i^2}{n}+ \frac{-A_n i}{n} + 2 \left(1+ \frac{1}{\sqrt{a_n}}\right) \frac{i}{n} +\Theta(\frac{A_n^2 i^3}{n^2}). 
\end{align*}
as long as $i=o\Big(\sqrt[3]{\frac{n^2}{a_n}}\Big)$ since $A_n =\Theta(\sqrt{a_n}).$ Another observation is that the first term is the dominant one provided that $a_n=o(n^2).$  \\

The case for  $k < k^*$ is similar. Taking $k=k^*-i$,  we have 
\begin{equation*}
R =  \frac{\binom{n}{k}^2 \frac{1}{a_n^{k}}}{\binom{n}{k^*}^2 \frac{1}{a_n^{k^*}}} 
 =  \frac{\binom{n}{k^*-i}^2 \frac{1}{a_n^{k^*-i}}}{\binom{n}{k^*}^2 \frac{1}{a_n^{k^*}}} 
 =  a_n^i \frac{(k^*)^2 \cdots (k^*-i+1)^2}{(n-k^*+i)^2 \cdots (n-k^*+1)^2}.
\end{equation*}
Then, similar computations yield
\begin{align*}
\ln R =& i \ln a_n + 2 \sum_{j=1}^i \ln \Big( \frac{k^*-j+1}{n-k^*+j}  \Big) \\
 \sim & -2A_n \frac{i(i+1)}{2n} + 2 (1+\sqrt{a_n}) \frac{i}{n} +\Theta(\frac{A_n^2 i^3}{n^2}) \\
=& \frac{-A_n i^2}{n}+ \frac{-A_n i}{n} + 2 (1+\sqrt{a_n}) \frac{i}{n} +\Theta(\frac{A_n^2 i^3}{n^2}).
\end{align*}
Altogether, we have
\begin{equation*}
\binom{n}{k}^2 \frac{1}{a_n^{k}} \sim \binom{n}{k^*}^2 \frac{1}{a_n^{k^*}} e^{\frac{-A_n i^2}{n}}.
\end{equation*}
where $k=k^* \pm i$, $a_n=o(n^2)$ and $i=o\Big(\sqrt[3]{\frac{n^2}{a_n}}\Big).$ We parametrize 
\begin{equation} \label{param}
k=k^{*} \pm  c \sqrt{\frac{n}{A_n}}
\end{equation}
where $c$ is a constant. Therefore, we have
\begin{equation*}
\binom{n}{k}^2 \frac{1}{a_n^{k}} \sim \binom{n}{k^*}^2 \frac{1}{a_n^{k^*}} e^{-c^2}.
\end{equation*}
Given the restriction $i=o\Big(\sqrt[3]{\frac{n^2}{a_n}}\Big)$, summing over the expression above all $c \in \mathbb{R}$ to find the asymptotics of the sum does not seem accurate at first glance. In order to see that it gives the correct asymptotics, we find an appropriate range for $k$ where the sum of terms is in agreement with the sum of all terms asymptotically. Similar to the argument in \cite{S14}, consider $[k^{-}, k^{+}]=[k^{*}- 2 \sqrt{\frac{n \ln n}{A_n}}, k^{*}+ 2 \sqrt{\frac{n \ln n}{A_n}}].$ Since $ \sqrt{\frac{n \ln n}{A_n}} = o\left(\sqrt[3]{\frac{n^2}{a_n}}\right)$, we have
\begin{equation*}
\binom{n}{k^{+}}^2 \frac{1}{a_n^{k^{+}}} \sim \binom{n}{k^*}^2 \frac{1}{a_n^{k^*}} n^{-4},
\end{equation*}  
and
\begin{equation*}
\sum_{l \geq k^{+}} \binom{n}{l}^2 \frac{1}{a_n^{l}}= o \left(\binom{n}{k^*}^2 \frac{1}{a_n^{k^*}} n^{-3} \right).
\end{equation*}
Exactly the same argument for $k^{-}$ allows us to conclude that the sum of the terms out of the range is negligible compared to the maximum term. So, in $[k^{-}, k^{+}]$, according to the parametrization \eqref{param}, we have
\begin{equation*}
\sum_{k=1}^{n} \binom{n}{k}^2 \frac{1}{a_n^k} \sim  \binom{n}{k^*}^2 \frac{1}{a_n^{k^*}} \sum_{k=1}^n  e^{\frac{-A_n (k-k^*)^2}{n}}  
 =   \binom{n}{k^*}^2 \frac{1}{a_n^{k^*}} \sum_{k}   e^{-c^2}.
\end{equation*}
 We can approximate the sum by the integral below.
\begin{eqnarray*}
 \binom{n}{k^*}^2 \frac{1}{a_n^{k^*}} \sum_{i=-k^* + 1}^{n-k^*}   e^{-c^2} & \sim   & \binom{n}{k^*}^2 \frac{1}{a_n^{k^*}} \int_{- \infty}^{\infty}  e^{-c^2} \sqrt{\frac{n}{A_n}} \, dc \\
 &= & \binom{n}{k^*}^2 \frac{1}{a_n^{k^*}} \sqrt{\frac{\pi n }{A_n}}. 
 \end{eqnarray*}
 Therefore, we have
 \begin{equation} \label{intapp}
 \mathbb{E}[T_n] =\sum_{k=1}^{n} \binom{n}{k}^2 \frac{1}{a_n^k} \sim \binom{n}{k^*}^2 \frac{1}{a_n^{k^*}} \sqrt{\frac{\pi n }{A_n}}.
\end{equation}
 
Finally, we evaluate $\binom{n}{k^*}$ asymptotically. We seperate into cases; each case corresponds to an interval on the order of $k^*,$ which is related to the order of $a_n$ through $k^*=\Theta(\frac{n}{\sqrt{a_n}}).$ We discuss only the case  where $a_n$ is a constant. All cases except the constant case are analyzed in \cite{S14} as much in detail as we need.\\

Suppose $a_n=a.$ Let $c=\frac{1}{1+\sqrt{a}}$. 
 Since $k^*=cn$ is linear in 
$k$,  we can apply Stirling's formula to obtain  
\begin{align*}
\binom{n}{k^*}=\binom{n}{cn} \sim & \frac{(n/e)^n}{(cn/e)^{cn} ((1-c)n/e)^{(1-c)n}} \frac{\sqrt{2\pi n}}{\sqrt{2 \pi n c} \sqrt{2 \pi n (1-c)}} \\
=& \frac{1}{c^{nc}(1-c)^{(1-c)n}} \frac{1}{\sqrt{2 \pi n c (1-c)}}.
\end{align*} 
For simplicity, we may write the expression in terms of $a$ and $k^*$ as 
\begin{equation*}
\frac{(1+\sqrt{a})^{n+1}}{\sqrt{2 \pi n} \sqrt{a}^{k^*+1}}.
\end{equation*}
Combining with the results in \cite{S14}, we list the limiting behavior of $\binom{n}{k^*}$ corresponding to different orders of $a_n$ in the table below.
\renewcommand{\arraystretch}{1.5}
\begin{center}
\begin{tabular}{|c|c|c|}
\hline
$\alpha(a_n=an^{\alpha})$ &  $\beta (k^*=\mathcal{O}(n^{\beta}))$ & $\binom{n}{k^*} \sim  $ \\ \hline 
$0$ & $1$ & $\frac{(1+\sqrt{a})^{n+1}}{\sqrt{2 \pi n} \sqrt{a}^{k^*+1}}$ \\  \hline
 $(0,1/2)$ & $(3/4,1)$ & $e^{\frac{-k^*{^2}}{2n}(1+o(1))}\frac{n^{k^*}}{k^*!}$  \\ \hline   
 $[1/2, 2/3)$  & $(2/3,3/4)$ & $e^{-k^*{^2}/2n}e^{-k^*{^3}/6n^2}\frac{n^{k^*}}{k^*!}$ \\  \hline
 $[2/3, 1)$  & $(1/2,2/3)$ & $e^{-k^*{^2}/2n}\frac{n^{k^*}}{k^*!}$ \\  \hline
 $1$ & $1/2$ & $e^{-1/2a}\frac{n^{k^*}}{k^*!}$   \\  \hline
$(1,2)$  & $(0,1/2)$ & $\frac{n^{k^*}}{k^*!}$  \\ \hline 
\end{tabular} 
\end{center}

Thus, we can rewrite \eqref{intapp} more explicitly as 
\begin{equation*}
\mathbb{E}[T_n] \sim \frac{\sqrt{\pi n}\sqrt[4]{an^{\alpha}}}{1 + \sqrt{an^{\alpha}}} e^{2\kappa(\alpha)} \frac{n^{2\frac{n}{1+ \sqrt{an^{\alpha}}}}}{ \Big[ \Big( \frac{n}{1+ \sqrt{an^{\alpha}}} \Big)! \Big]^2} \frac{1}{(an^{\alpha})^{\frac{n}{1+ \sqrt{an^{\alpha}}}}},
\end{equation*}
where $\kappa(\alpha)$ is the exponent in the last column of the table above given by \[   \kappa(\alpha)=\left\{
\begin{array}{ll}
      -\frac{k^*{^2}}{2n}(1+o(1)) & 0 < \alpha < 1/2 \\
      -\frac{k^*{^2}}{2n}-\frac{k^*{^3}}{6n^2} & 1/2 \leq \alpha < 2/3 \\
      -\frac{k^*{^2}}{2n} & 2/3 \leq \alpha < 1 \\
      -\frac{1}{2a} &  \alpha =1  \\
      0 & 1 < \alpha < 2 \\
\end{array} 
\right. \]
Then, applying Stirling's formula to the factorial in the denominator above, and after cancellations, we eventually have
\begin{equation*}
\frac{\sqrt[4]{an^{\alpha}}}{2 \sqrt{\pi n}} e^{ 2\kappa(\alpha)} e^{\frac{2n}{1+ \sqrt{an^{\alpha}}}} \left(1+ \frac{1}{\sqrt{an^{\alpha}}} \right)^{\frac{2n}{1+ \sqrt{an^{\alpha}}}}.
\end{equation*}
 Further simplifications of the expression for $\alpha \in (2/3,2),$ which follows from Lemma \ref{calc} below, conclude the proof.  
 \end{proof}
\begin{lemma} \label{calc}
If $\alpha \in (2/3,2),$ then 
\begin{equation*}
\left(1+ \frac{1}{\sqrt{an^{\alpha}}} \right)^{\frac{2n}{1+ \sqrt{an^{\alpha}}}} \sim e^{\frac{2n}{an^{\alpha}+ \sqrt{an^{\alpha}}}}.
\end{equation*}
In particular, if $\alpha \in (1,2)$, then
\begin{equation*}
\left(1+ \frac{1}{\sqrt{an^{\alpha}}} \right)^{\frac{2n}{1+ \sqrt{an^{\alpha}}}} \sim 1.
\end{equation*}
\end{lemma}

\begin{proof}
Define 
\begin{equation*}
f(n) = \frac{\left(1+ \frac{1}{\sqrt{an^{\alpha}}}\right)^{\sqrt{an^{\alpha}}}}{e}
\end{equation*}  and 

\begin{equation*}
g(n)=\frac{2n}{an^{\alpha}+ \sqrt{an^{\alpha}}}.
\end{equation*}
We can equivalently show that $\lim_{n \rightarrow \infty} [f(n)]^{g(n)} = 1$ for $\alpha \in (2/3,2).$ The proof relies on elementary techniques. First, evaluate the limit of the logarithm of the expression. We have 
\begin{align*}
\lim_{n \rightarrow \infty} \ln L =& \lim_{n \rightarrow \infty} g(n) \ln f(n) \\
=& \lim_{n \rightarrow \infty} \frac{2n}{an^{\alpha}+ \sqrt{an^{\alpha}}}\left( \sqrt{an^{\alpha}} \ln\left(1+ \frac{1}{\sqrt{an^{\alpha}}}\right) - 1 \right)\\
\end{align*}
Then L'Hopital's Rule gives
\begin{equation*}
\lim_{n \rightarrow \infty} \sqrt{an^{\alpha}} \ln\left(1+ \frac{1}{\sqrt{an^{\alpha}}}\right)=1.
\end{equation*}
If $\alpha \in(1,2),$ then $\lim_{n \rightarrow \infty} \ln L =0;$ therefore the claim is true for this case. \\
Now suppose $\alpha <1.$ We apply L'Hopital's rule to $\ln L$ one more time, which gives
\begin{eqnarray*}
\lim_{n \rightarrow \infty} \ln L &=& \lim_{n \rightarrow \infty} \frac{2 \sqrt{a}\alpha \ln\Big(1+ \frac{1}{\sqrt{an^{\alpha}}}\Big) n^{\alpha/2-1} - 2 \alpha \frac{1}{1+ \frac{1}{\sqrt{an^{\alpha}}}}n^{-1}}{\frac{2a(1- \alpha)}{2}n^{\alpha-2} + \frac{\sqrt{a}}{2}\Big(\frac{\alpha}{2}-1 \Big) n ^{\frac{\alpha}{2}-2}} \\
&=& \lim_{n \rightarrow \infty} C \Big(\sqrt{a n^{\alpha}} \ln\Big(1+ \frac{1}{\sqrt{an^{\alpha}}}\Big) - \frac{1}{1+ \frac{1}{\sqrt{an^{\alpha}}}}\Big) n^{1- \alpha} \\
\end{eqnarray*}
for some constant $C.$ Then we consider the Taylor expansion of the expression. Define $u=\sqrt{a n^{\alpha}} $. We have
\begin{align*}
\lim_{n \rightarrow \infty} \ln L  = & \lim_{u \rightarrow \infty} C \Big(u \ln \Big(1+ \frac{1}{u}\Big) - \frac{1}{1+ \frac{1}{u}} \Big) u^{\frac{2}{\alpha}- 2} \\
=& C \Big( u\sum_{i=1}^{\infty} \frac{(-1)^{i+1} u^{-i}}{i} -\sum_{i=0}^{\infty}     (-1)^i u^{-i}\Big)u^{\frac{2}{\alpha}- 2} \\
=& C   \Big( \sum_{i=1}^{\infty} \Big( \frac{(-1)^i}{i+1} - (-1)^i \Big) u ^{-i}            \Big) u^{\frac{2}{\alpha}- 2} \\
=& C \Big( \frac{1}{2u} - \frac{2}{3 u^2} + \frac{3}{4u^3}- \cdots \Big)u^{\frac{2}{\alpha}- 2}. \\
\end{align*}
Hence, $\lim_{n \rightarrow \infty} \ln L$ is zero as long as $\frac{2}{\alpha}- 2$ is less than one. Then the result follows.

\end{proof}

\bigskip

\acknowledgements The first author is supported by  the Scientific
and Research Council of Turkey [TUBITAK-117C047]. We would like to thank Michael Waterman for providing the algorithm in Chapter 5 and correcting the order of the computation time of the algorithm in the final version of the paper.

\nocite{*}
\bibliographystyle{abbrvnat}
\bibliography{steele}
\label{sec:biblio}

\end{document}